\documentclass[dvipsnames]{amsart}
\usepackage{amsmath,amssymb,amsthm,url,xcolor}
\usepackage{xcolor}
\definecolor{yellow2}{cmyk}{0, 0.18, 0.74, 0.18}
\definecolor{green2}{cmyk}{0.58, 0, 0.99, 0.01}

\newtheorem{thr}{Theorem}

\newtheorem{cor}[thr]{Corollary}

\newtheorem{observation}[thr]{Observation}

\theoremstyle{definition}

\theoremstyle{remark}

\numberwithin{equation}{section}

\begin{document}

\title[Nonnegative rank: Hard problems, easy solutions]{The nonnegative rank of a matrix: \\ Hard problems, easy solutions}

\author{Yaroslav Shitov}
\address{Higher School of Economics, 20 Myasnitskaya Ulitsa, Moscow 101000, Russia}
\email{yaroslav-shitov@yandex.ru}






\begin{abstract}
Using elementary linear algebra, we develop a technique that leads to solutions of two widely known problems on nonnegative matrices. First, we give a short proof of the result by Vavasis stating that the nonnegative rank of a matrix is NP-hard to compute. This proof is essentially contained in the paper by Jiang and Ravikumar, who discussed this topic in different terms fifteen years before the work of Vavasis. Secondly, we present a solution of the Cohen--Rothblum problem on rational nonnegative factorizations, which was posed in 1993 and remained open.
\end{abstract}

\maketitle

\section{Introduction}

Let $A$ be a matrix with nonnegative real entries. The \textit{nonnegative rank} of $A$ is the smallest $k$ for which there exist $k$ nonnegative rank-one matrices that sum to $A$. This concept is related to the \textit{nonnegative matrix factorization} problem, which is the task to find the optimal approximation of a given matrix with a matrix of given nonnegative rank. This problem has important applications in different branches of modern science, including data mining~\cite{LS}, combinatorial optimization~\cite{Yan}, statistics~\cite{KRS}, quantum mechanics~\cite{CR}, and many others.

This paper presents a combinatorial approach that leads to very short solutions of two widely known problems on nonnegative matrices, and one of these accomplishments is a proof that the nonnegative rank is NP-hard to compute. The paper~\cite{Vavas} by Vavasis has become a standard reference to this result in applied mathematics, and the proof it contains is based on ingenious geometric considerations. The proof we present is short and does not require any special knowledge. When this writing was about to be completed, the author learned that a similar proof is contained in the paper~\cite{JR} by Jiang and Ravikumar. Their Lemma~3.3 is stated in different terms, but it is essentially the oldest reference to the NP-hardness of nonnegative rank we are aware of. 

Another result that we obtain is a solution of the \textit{Cohen--Rothblum problem}. Assume that $A$ is a rational matrix with nonnegative rank $k$, do there exist $k$ \textit{rational} nonnegative rank-one matrices that sum to $A$? This problem was posed in the foundational paper~\cite{CR} more than 20 years ago, and it has been widely discussed in the literature. Vavasis~\cite{Vavas} demonstrates the connection between the Cohen--Rothblum problem and algorithmic complexity of nonnegative rank. Another notable application of nonnegative ranks is the theory of extended formulations of polytopes~\cite{Yan}, and the possible lack of optimal rational factorizations is a difficulty in this theory~\cite{Roth}. Kubjas, Robeva, and Sturmfels~\cite{KRS} consider this problem in context of modern statistics; they also give a partial solution of this problem. We note the paper~\cite{GFR}, which contains a solution of a similar problem but asked for \textit{positive semidefinite} rank. Our paper contains an explicit example of a $21\times 21$ matrix with integral entries that can be written as a sum of $19$ nonnegative rank-one matrices but not as a sum of $19$ rational nonnegative rank-one matrices. This gives a solution of the Cohen--Rothblum problem, which has been open until now. (Our solution of the Cohen--Rothblum problem appears on arXiv as the manuscript~\cite{myCR}, which has been uploaded concurrently with another solution of the problem, see~\cite{anotherproof}. Also, the earlier paper~\cite{myNRDF} contains a related result that is both more general and more specific: There is a real matrix $A$ of conventional rank five which admits a nonnegative rank factorization but no such factorization is rational in the entries of $A$.)

\section{Our technique}

Let $A$ be a nonnegative matrix with entries in a subfield $\mathbb{F}\subset\mathbb{R}$. The \textit{nonnegative rank} of $A$ \textit{with respect to} $\mathbb{F}$ (denoted $\operatorname{rank}_{\mathbb{F}+} A$) is the smallest $k$ for which $A$ is a sum of $k$ nonnegative rank-one matrices over $\mathbb{F}$. In particular, the quantity $\operatorname{rank}_{\mathbb{R}+}(A)$ is the usual nonnegative rank as defined in the introduction.
Combinatorial methods of studying the nonnegative rank have a long history, see the recent survey~\cite{FKPT}, the older papers~\cite{GP, Orl}, and references therein. Methods that employ zero patterns of matrices are still developing and 
being used in modern investigations~\cite{CCZ, FMPTdW}. As we will see later, the combinatorial analysis of zero patterns plays a crucial role in our approach. 

Let $\alpha_1,\ldots,\alpha_n$ be nonnegative real numbers. The $5\times(n+4)$ matrix
$$\mathcal{B}(\alpha_1,\ldots,\alpha_n)=\begin{pmatrix}
\alpha_1&\ldots&\alpha_n&1&1&1&1\\
1&\ldots&1&1&1&0&0\\
0&\ldots&0&0&1&1&0\\
0&\ldots&0&0&0&1&1\\
0&\ldots&0&1&0&0&1
\end{pmatrix}$$
will be particularly important in our consideration. Its lower-right $4\times4$ submatrix
$$\mathcal{B}_0=
\begin{pmatrix}
1&1&0&0\\
0&1&1&0\\
0&0&1&1\\
1&0&0&1
\end{pmatrix}
$$
appears in the note~\cite{Thomas}, which is one of the oldest references on the topic of nonnegative matrix factorizations. The reason why $\mathcal{B}_0$ is so important is that it has different conventional and nonnegative ranks, --- one has $\operatorname{rank} \mathcal{B}_0=3$ and $\operatorname{rank}_+ \mathcal{B}_0=4$. In particular, the row of ones can be represented as a linear combination of the rows of $\mathcal{B}_0$ in different ways: one can take the sum of the first and third rows, or the sum of the second and fourth rows, or any convex combination of these two sums. This leads us to the following easy but important observation.

\begin{observation}\label{obs1}
$\operatorname{Rank}_{\mathbb{F}+}\mathcal{B}(\alpha_1,\ldots,\alpha_n)=4$ if and only if $\alpha_1=\ldots=\alpha_n\in[0,1]\cap\mathbb{F}$.
\end{observation}

In other words, this matrix has the following interesting property. For any $\varepsilon\in(0,1)$, the $5\times5$ matrices obtained from $\mathcal{B}(\varepsilon)$ by a small perturbation of the $(1,1)$ entry have the same nonnegative rank as the $(1,1)$ cofactor of $\mathcal{B}(\varepsilon)$. Basic results of linear algebra show that such a property cannot hold for a conventional rank function of matrices over a field. In fact, the usual rank of a matrix equals the order of the largest non-singular submatrix, so a small perturbation of an entry would lead to a matrix with rank greater than the rank of its cofactor. We use this distinction between the conventional and nonnegative rank functions in the following matrix completion problem: Given nonnegative matrices $A\in\mathbb{F}^{m\times n}$, $B\in\mathbb{F}^{m\times k}$, $c\in\mathbb{F}^{1\times n}$ and a subset $I\subset\mathbb{F}$, what is the smallest nonnegative rank of
\begin{equation}\label{eqAx}
\mathcal{A}(x)=\left(\begin{array}{c|c}
A&B\\\hline
c&\textcolor{blue}{}{x\ldots x}
\end{array}\right)
\end{equation}
provided that $x\in I$? Namely, we can reduce this completion version to the standard formulation of the nonnegative rank problem.

\begin{thr}\label{pr3}
Let $A\in\mathbb{F}^{m\times n}$, $B\in\mathbb{F}^{m\times k}$, $c\in\mathbb{F}^{1\times n}$ be nonnegative matrices, let $r,s$ be positive integers. We assume that either $\operatorname{rank}_{\mathbb{F}+} A\geqslant r$ or $k=1$. We define 
$$
\mathcal{G}=
\left(\begin{array}{c|c|cccccc}
&&0&0&0&0\\
{A}&{B}&\vdots&\vdots&\vdots&\vdots\\
&&0&0&0&0\\\hline
{c}&\textcolor{blue}{s\ldots s}&\textcolor{green2}{1}&\textcolor{green2}{1}&\textcolor{green2}{1}&\textcolor{green2}{1}\\\hline
0\ldots0&\textcolor{green2}{1\ldots1}&\textcolor{red}{1}&\textcolor{green2}{1}&\textcolor{green2}{0}&\textcolor{green2}{0}\\
0\ldots0&\textcolor{green2}{0\ldots0}&\textcolor{green2}{0}&\textcolor{red}{1}&\textcolor{green2}{1}&\textcolor{green2}{0}\\
0\ldots0&\textcolor{green2}{0\ldots0}&\textcolor{green2}{0}&\textcolor{green2}{0}&\textcolor{red}{1}&\textcolor{green2}{1}\\
0\ldots0&\textcolor{green2}{0\ldots0}&\textcolor{green2}{1}&\textcolor{green2}{0}&\textcolor{green2}{0}&\textcolor{red}{1}
\end{array}\right)
$$
and $\mathcal{A}(x)$ as in~\eqref{eqAx}. Then the inequality $\operatorname{rank}_{\mathbb{F}+}\mathcal{G}\leqslant r+4$ holds if and only if there is a $\xi\in[s-1,s]\cap\mathbb{F}$ such that $\operatorname{rank}_{\mathbb{F}+}\mathcal{A}(\xi)\leqslant r$.
\end{thr}

\begin{proof}
Note that $\mathcal{G}$ is the sum of $\mathcal{A}(\xi)$ and $\mathcal{B}(s-\xi,\ldots,s-\xi)$ up to adding zero rows and columns to the latter matrices. Since $\operatorname{rank}_{\mathbb{F}+}\mathcal{B}(s-\xi,\ldots,s-\xi)=4$ by Observation~\ref{obs1}, the 'if' direction follows immediately.

Let us prove the 'only if' direction. First, we note that the inequalities
\begin{equation}\label{eq1}\operatorname{rank}_{\mathbb{F}+} \left(A\left|\right. B\right)\geqslant r\mbox{$ $ $ $ and $ $ $ $} \operatorname{rank}_{\mathbb{F}+} \begin{pmatrix}A\\\hline c\end{pmatrix}\geqslant r\end{equation}
hold trivially if $\operatorname{rank}_{\mathbb{F}+} A\geqslant r$. (We note that the block matrices involved in~\eqref{eq1} have sizes $m\times(n+k)$ and $(m+1)\times n$, respectively.) If $k=1$, then $\mathcal{A}(x)$ can be obtained by adding a row (a column, respectively) to the matrix in the left (right, respectively) inequality of~\eqref{eq1}. Therefore, if~\eqref{eq1} were not true, we would get $\operatorname{rank}_{\mathbb{F}+}\mathcal{A}(x)\leqslant r$ and complete the proof. So we can assume that the inequalities~\eqref{eq1} are true.

Now we assume that $G_1,\ldots,G_{r+4}$ are nonnegative rank-one matrices that sum to $\mathcal{G}$. If there were a $G_i$ with two non-zero red entries, then $\mathcal{G}$ would have a $2\times 2$ submatrix with positive entries two of which are red, --- but this is not the case. Therefore, any $G_i$ contains at most one non-zero red entry, so there are at least four $G_i$'s with non-zero red entries. We call these $G_i$'s red, denote their sum by $\textcolor{red}{R}$, and observe that they have zero black entries. Using the first inequality in~\eqref{eq1}, we get that there are $r$ non-red $G_i$'s, and every of them has a non-zero entry either in the block corresponding to $A$ or in the block corresponding to $B$. Using the second inequality in~\eqref{eq1}, we see that every non-red $G_i$ contains a non-zero entry either in the block corresponding to $A$ or in the block corresponding to $c$. The two previous sentences allow us to conclude that the non-red $G_i$'s do not contribute to the green entries of $\mathcal{G}$. Taking into account Observation~\ref{obs1}, we get that, for some $a\in[0,1]$, the matrix $\textcolor{red}{R}$ equals $\mathcal{B}(a,\ldots,a)$ up to the zeroes at black positions. Therefore, the top-left $(m+1)\times(n+k)$ submatrix of $\mathcal{G}-\textcolor{red}{R}$ is the matrix $\mathcal{A}(s-a)$ as in~\eqref{eqAx}, which completes the proof.
\end{proof}

\begin{cor}\label{pr33}
Let $\mathcal{A}(x)$, $\mathcal{G}$ be as in Theorem~\ref{pr3}. If $k=1$, then we have
$$\operatorname{rank}_{\mathbb{F}+}\mathcal{G}=\min \operatorname{rank}_{\mathbb{F}+}\mathcal{A}(\xi)+4,$$
where the minimum is taken over all $\xi\in[s-1,s]\cap\mathbb{F}$.
\end{cor}

\section{On the complexity of nonnegative rank}

Corollary~\ref{pr33} allows us to construct a polynomial time reduction to nonnegative rank from the following problem, which is known (see~\cite{Watson}) as the \textit{nonnegative rank of a partial 0-1 matrix}. Given a matrix $X$ whose $(i,j)$ entry can be $0,1$ or $x_{ij}$; what is the smallest nonnegative rank of matrices that are obtained from $X$ by assigning numbers in $[0,1]$ to the variables $x_{ij}$? We are going to prove that this problem is NP-hard. 

Let $G$ be a simple graph with vertex set $V$ and edge set $E$. A \textit{clique} in $G$ is a subset $U\subset V$ such that $\{u_1,u_2\}\in E$ for all distinct $u_1,u_2\in U$. The \textit{clique covering number} of $G$ is the smallest number of cliques needed to cover $V$.
Let $X=X(G)$ be the matrix (whose rows and columns are indexed with vertices of $G$) defined by $X_{vv}=1$, $X_{uv}=x_{uv}$ if $u,v$ are adjacent, and $X_{uv}=0$ otherwise.

Let $U_1,\ldots,U_c$ be cliques whose union is $V$, and we assume without loss of generality that these cliques are disjoint. For all $i\in\{1,\ldots,c\}$, we define the matrix $H^i$ as $H^i_{\alpha\beta}=1$ if $\alpha,\beta\in U_i$ and $H^i_{\alpha\beta}=0$ otherwise. Clearly, the matrix $H_1+\ldots+H_c$ has nonnegative rank at most $c$ and can be obtained from $X$ by replacing the variables with zeros and ones.


Conversely, let $M^1,\ldots,M^r$ be nonnegative rank-one matrices whose sum can be obtained from $X$ by replacing the variables with numbers in $[0,1]$. If one of the sets $V^l=\{v\in V\left|\right.M^l_{vv}>0\}$ was not a clique, we would find non-adjacent distinct vertices $u,v$ in $V^l$. This would imply $X_{uv}=0$ and contradict $M^l_{uv}>0$. So we see that $V^1,\ldots,V^r$ are cliques, and their union is $V$. This proves that the clique covering number of $G$ equals the smallest nonnegative rank of matrices that are obtained from $X(G)$ by replacing the variables with numbers in $[0,1]$. Therefore, since the clique covering number is NP-hard to compute (see~\cite{Karp}), so is the nonnegative rank for partial 0-1 matrices. The observation in the beginning of this section shows that the standard version of the nonnegative rank problem is NP-hard as well.

As pointed out in the introduction, Lemma~3.3 in~\cite{JR} contains essentially the same result. Namely, Jiang and Ravikumar prove the NP-hardness of the \textit{normal set basis problem}, which is a reformulation of the more commonly known \textit{biclique partition number} problem. We refer the reader to~\cite{CHHK} for a discussion of these two formulations and further complexity results. There is also a straightforward correspondence between the biclique partition number of a bipartite graph and the nonnegative rank of its adjacency matrix, see Remark~6.4 in~\cite{Orl} and Lemma~2.2 in~\cite{GP}. We hope that these references can help an interested reader to translate the proof of Lemma~3.3 in~\cite{JR} into the language of matrix theory and come up with the NP-hardness proof for nonnegative rank.

\section{Nonnegative rank depends on the field}

We proceed with a solution of the Cohen--Rothblum problem. To this end, we consider the matrix
$$
\begin{pmatrix}
\textcolor{magenta}{2}&\textcolor{yellow2}{2}&\textcolor{yellow2}{2}&\textcolor{yellow2}{1}&\textcolor{yellow2}{0}&0&0&0&0&0&0&0&0&0&0&0&0&\textcolor{magenta}{1}&\textcolor{magenta}{1}&\textcolor{magenta}{1}&\textcolor{magenta}{1}\\
\textcolor{yellow2}{1}&\textcolor{yellow2}{2}&\textcolor{yellow2}{1}&\textcolor{yellow2}{0}&\textcolor{yellow2}{1}&0&0&0&0&0&0&0&0&0&0&0&0&0&0&0&0\\
\textcolor{yellow2}{0}&\textcolor{yellow2}{0}&\textcolor{yellow2}{1}&\textcolor{green2}{2}&\textcolor{yellow2}{0}&0&0&0&0&0&0&0&0&\textcolor{green2}{1}&\textcolor{green2}{1}&\textcolor{green2}{1}&\textcolor{green2}{1}&0&0&0&0\\
\textcolor{yellow2}{0}&\textcolor{yellow2}{1}&\textcolor{yellow2}{0}&\textcolor{yellow2}{0}&\textcolor{blue}{2}&0&0&0&0&\textcolor{blue}{1}&\textcolor{blue}{1}&\textcolor{blue}{1}&\textcolor{blue}{1}&0&0&0&0&0&0&0&0\\
\textcolor{yellow2}{0}&\textcolor{yellow2}{1}&\textcolor{yellow2}{1}&\textcolor{red}{2}&\textcolor{red}{2}&\textcolor{red}{1}&\textcolor{red}{1}&\textcolor{red}{1}&\textcolor{red}{1}&0&0&0&0&0&0&0&0&0&0&0&0\\
0&0&0&\textcolor{red}{1}&\textcolor{red}{1}&\textcolor{red}{1}&\textcolor{red}{1}&\textcolor{red}{0}&\textcolor{red}{0}&0&0&0&0&0&0&0&0&0&0&0&0\\
0&0&0&\textcolor{red}{0}&\textcolor{red}{0}&\textcolor{red}{0}&\textcolor{red}{1}&\textcolor{red}{1}&\textcolor{red}{0}&0&0&0&0&0&0&0&0&0&0&0&0\\
0&0&0&\textcolor{red}{0}&\textcolor{red}{0}&\textcolor{red}{0}&\textcolor{red}{0}&\textcolor{red}{1}&\textcolor{red}{1}&0&0&0&0&0&0&0&0&0&0&0&0\\
0&0&0&\textcolor{red}{0}&\textcolor{red}{0}&\textcolor{red}{1}&\textcolor{red}{0}&\textcolor{red}{0}&\textcolor{red}{1}&0&0&0&0&0&0&0&0&0&0&0&0\\
0&0&0&0&\textcolor{blue}{1}&0&0&0&0&\textcolor{blue}{1}&\textcolor{blue}{1}&\textcolor{blue}{0}&\textcolor{blue}{0}&0&0&0&0&0&0&0&0\\
0&0&0&0&\textcolor{blue}{0}&0&0&0&0&\textcolor{blue}{0}&\textcolor{blue}{1}&\textcolor{blue}{1}&\textcolor{blue}{0}&0&0&0&0&0&0&0&0\\
0&0&0&0&\textcolor{blue}{0}&0&0&0&0&\textcolor{blue}{0}&\textcolor{blue}{0}&\textcolor{blue}{1}&\textcolor{blue}{1}&0&0&0&0&0&0&0&0\\
0&0&0&0&\textcolor{blue}{0}&0&0&0&0&\textcolor{blue}{1}&\textcolor{blue}{0}&\textcolor{blue}{0}&\textcolor{blue}{1}&0&0&0&0&0&0&0&0\\
0&0&0&\textcolor{green2}{1}&0&0&0&0&0&0&0&0&0&\textcolor{green2}{1}&\textcolor{green2}{1}&\textcolor{green2}{0}&\textcolor{green2}{0}&0&0&0&0\\
0&0&0&\textcolor{green2}{0}&0&0&0&0&0&0&0&0&0&\textcolor{green2}{0}&\textcolor{green2}{1}&\textcolor{green2}{1}&\textcolor{green2}{0}&0&0&0&0\\
0&0&0&\textcolor{green2}{0}&0&0&0&0&0&0&0&0&0&\textcolor{green2}{0}&\textcolor{green2}{0}&\textcolor{green2}{1}&\textcolor{green2}{1}&0&0&0&0\\
0&0&0&\textcolor{green2}{0}&0&0&0&0&0&0&0&0&0&\textcolor{green2}{1}&\textcolor{green2}{0}&\textcolor{green2}{0}&\textcolor{green2}{1}&0&0&0&0\\
\textcolor{magenta}{1}&0&0&0&0&0&0&0&0&0&0&0&0&0&0&0&0&\textcolor{magenta}{1}&\textcolor{magenta}{1}&\textcolor{magenta}{0}&\textcolor{magenta}{0}\\
\textcolor{magenta}{0}&0&0&0&0&0&0&0&0&0&0&0&0&0&0&0&0&\textcolor{magenta}{0}&\textcolor{magenta}{1}&\textcolor{magenta}{1}&\textcolor{magenta}{0}\\
\textcolor{magenta}{0}&0&0&0&0&0&0&0&0&0&0&0&0&0&0&0&0&\textcolor{magenta}{0}&\textcolor{magenta}{0}&\textcolor{magenta}{1}&\textcolor{magenta}{1}\\
\textcolor{magenta}{0}&0&0&0&0&0&0&0&0&0&0&0&0&0&0&0&0&\textcolor{magenta}{1}&\textcolor{magenta}{0}&\textcolor{magenta}{0}&\textcolor{magenta}{1}
\end{pmatrix}
$$
\noindent and denote it by $\mathcal{M}$. We are going to show that the rational and real nonnegative ranks of $\mathcal{M}$ are different. Our strategy is to apply Theorem~\ref{pr3} to $\mathcal{M}$, and we assume that $\mathbb{F}$ is a field as in Section~2. Let us remove the last twelve rows and columns from $\mathcal{M}$, replace the $(1,1)$, $(3,4)$, $(4,5)$ entries by the variables $a,b,c$, and denote the resulting matrix by $\mathcal{M}_1(a,b,c)$. The threefold application of Corollary~\ref{pr33} implies
\begin{equation}\label{eqpr1}\operatorname{rank}_{\mathbb{F}+}\mathcal{M}=\min\limits_{a,b,c\in[1,2]\cap\mathbb{F}}\,\operatorname{rank}_{\mathbb{F}+}\mathcal{M}_1(a,b,c)+12.\end{equation} Observing that the $3\times 3$ submatrix defined as the intersection of the second, third, and fourth rows and the first three columns of $\mathcal{M}$ has rank three, we apply Theorem~\ref{pr3} with $r=3$ to $\mathcal{M}_1(a,b,c)$. Taking into account~\eqref{eqpr1}, we get that the inequality $\operatorname{rank}_{\mathbb{F}+}\mathcal{M}\leqslant19$ holds if and only if there are $a,b,c,d\in[1,2]\cap\mathbb{F}$ such that the matrix
$$
\mathcal{C}(a,b,c,d)=\begin{pmatrix}
a&2&2&1&0\\
1&2&1&0&1\\
0&0&1&b&0\\
0&1&0&0&c\\
0&1&1&d&d
\end{pmatrix}
$$
has nonnegative rank at most three with respect to $\mathbb{F}$. Basic tools of linear algebra allow us to show that the conventional rank of $\mathcal{C}(a,b,c,d)$ exceeds three unless $b=c=d=1\pm\sqrt{0.5}$ and $a=2-1/b$. In particular, the rational nonnegative rank of $\mathcal{C}$ cannot be less than four, which implies $\operatorname{rank}_{\mathbb{Q}+}\mathcal{M}\geqslant20$. On the other hand, for $\alpha=1+\sqrt{0.5}$,
$$\begin{pmatrix}
0&0&0&0&0\\
0&\alpha^{-1}&0&0&1\\
0&0&0&0&0\\
0&1&0&0&\alpha\\
0&1&0&0&\alpha
\end{pmatrix}
+
\begin{pmatrix}
0&0&\alpha^{-1}&1&0\\
0&0&0&0&0\\
0&0&1&\alpha&0\\
0&0&0&0&0\\
0&0&1&\alpha&0
\end{pmatrix}
+\begin{pmatrix}
\sqrt{2}&2&\sqrt{2}&0&0\\
1&\sqrt{2}&1&0&0\\
0&0&0&0&0\\
0&0&0&0&0\\
0&0&0&0&0
\end{pmatrix}
$$
is a representation of $\mathcal{C}(\sqrt{2},\alpha,\alpha,\alpha,\alpha)$ as a sum of three nonnegative rank-one matrices. This shows that $\operatorname{rank}_{\mathbb{R}+}\mathcal{M}\leqslant19$ and completes the proof that the rational and real nonnegative ranks of $\mathcal{M}$ are different.

\section{Note added in proof}

The idea of my construction comes from the earlier paper~\cite{myBool}, which contains a similar NP-hardness proof but in the setting of Boolean matrices. As I have subsequently learned, Theorem 3.3 in~\cite{theirBool} has the same formulation and proof as the result in~\cite{myBool} but in the language of graph theory. I feel sorry that I did not know about~\cite{theirBool} when I was preparing~\cite{myBool} for publication, but such ocurrences seem to be ubiquitous in modern mathematics and hard to avoid. (As said above, a similar situation arose with Lemma 3.3 in~\cite{JR}, which states the NP-hardness of nonnegative rank but not in the terms common in applied mathematics.)

\section{Acknowledgements}

This work was carried out at National Research University --- Higher School of Economics in Moscow. I owe my gratitude to the university for their support of research activities and to my students, from whom I probably learned more than they from me.
I would like to thank the editors and anonymous referees of \textit{SIAM Review} for careful reading of the paper and helpful suggestions, which allowed me to improve the presentation of the results. Before I decided to include the solution of the Cohen--Rothblum problem in this paper, the manuscript~\cite{myCR} was undergoing the reviewing process in \textit{SIAM Journal of Applied Algebraic Geometry}. I am grateful to the referee of that journal for the detailed report and helpful comments.

\end{document}